\documentclass[11pt,a4paper,reqno]{amsart}
\usepackage[utf8]{inputenc}
\usepackage{amsmath, amsthm}
\usepackage{amsfonts}
\usepackage{amssymb}
\usepackage{graphicx}
\usepackage[left=2.5cm,right=2.5cm,top=2cm,bottom=2cm]{geometry}
\usepackage{amscd}
\usepackage[all]{xy}
\usepackage{hyperref}
\usepackage{gensymb}
\usepackage{enumerate}
\usepackage{csquotes}
\usepackage{comment}
\usepackage{hhline}
\usepackage{float}

\theoremstyle{plain}
\newtheorem{theorem}{Theorem}[section]
\newtheorem{proposition}[theorem]{Proposition}
\newtheorem{lemma}[theorem]{Lemma}
\newtheorem{corollary}[theorem]{Corollary}
\newtheorem{conjecture}[theorem]{Conjecture}

\begin{document}

\title{Self-similarity in the Kepler-Heisenberg problem}

\author{Victor Dods}
\email{victor.dods@gmail.com}

\author{Corey Shanbrom}
\address{California State University, Sacramento, 6000 J St., Sacramento, CA 95819, USA}
\email{corey.shanbrom@csus.edu}
\thanks{This material is based upon work supported by the National Science Foundation under Grant No. DMS-1440140 while the first author was in residence at the Mathematical Sciences Research Institute in Berkeley, California, during Fall 2018.}

\date{\today}



\begin{abstract}
The Kepler-Heisenberg problem is that of determining the motion of a planet around a sun in the Heisenberg group, thought of as a three-dimensional sub-Riemannian manifold. The sub-Riemannian Hamiltonian provides the kinetic energy, and the gravitational potential is given by the fundamental solution to the sub-Laplacian. The dynamics are at least partially integrable, possessing two first integrals as well as a dilational momentum which is conserved by orbits with zero energy. The system is known to admit closed orbits of any rational rotation number, which all lie within the fundamental zero-energy integrable subsystem. Here we demonstrate that, under mild conditions, zero-energy orbits are self-similar. Consequently we find that these zero-energy orbits stratify into three families: future collision, past collision, and quasi-periodicity without collision. If a collision occurs, it occurs in finite time.
\end{abstract}

\maketitle


\section{Introduction}\label{intro}

In geometric mechanics one usually constructs a dynamical system on the cotangent bundle of a Riemannian manifold $(M, g)$ by taking a Hamiltonian of the form $H=K+U$, where the kinetic energy $K$ is determined by the metric $g$ and the potential energy $U$ is chosen to represent a particular physical system.  In particular, the classical Kepler problem has been extensively studied in spaces of constant curvature. 
These investigations date back to the discovery of non-Euclidean geometry, when both Lobachevsky and Bolyai independently posed the Kepler problem in hyperbolic 3-space.  Serret, Lipschitz, and Killing studied the Kepler problem on two and three dimensional spheres.  Modern researchers consider $n$ bodies in spaces of arbitrary constant curvature $\kappa$.  
See \cite{Diacu} and the references therein for a thorough history. 
In \cite{MS}, we first posed the Kepler problem on the Heisenberg group in the following manner.

Let $(\mathcal H, D, \langle \cdot, \cdot\rangle)$ denote the sub-Riemannian geometry of the \emph{Heisenberg group}:
\begin{itemize}
\item $\mathcal H$ is diffeomorphic to $\mathbb R^3$ with usual global coordinates $(x, y, z)$
\item $D$ is the plane field distribution spanned by the vector fields $X := \frac{\partial}{\partial x} -\frac{1}{2}y\frac{\partial}{\partial z}$ and $Y := \frac{\partial}{\partial y} +\frac{1}{2}x\frac{\partial}{\partial z}$ 
\item $\langle \cdot, \cdot\rangle$ is the inner product on $D$ which makes $X$ and $Y$ orthonormal; that is, $ds^2 = (dx^2+dy^2)|_D$.
\end{itemize} 
Curves tangent to the distribution have the property that their $z$-coordinate always equals the area traced out by their projection to the $x,y$-plane. See \cite{Tour} for a detailed description of this geometry.

We define the Kepler problem on the Heisenberg group to be the dynamical system on $T^*\mathcal H=(x, y, z, p_x, p_y, p_z)$ with Hamiltonian 
$$H=\underbrace{\tfrac{1}{2}((p_x-\tfrac{1}{2}yp_z)^2+(p_y+\tfrac{1}{2}xp_z)^2)}_K\underbrace{-\frac{1}{8\pi\sqrt{(x^2+y^2)^2+{16}z^2}}}_{U}.$$
The kinetic energy is $K=\frac{1}{2}(P_X^2 +P_Y^2)$, where $P_X= p_x-\frac{1}{2}yp_z$ and $P_Y= p_y+\frac{1}{2}xp_z$ are the dual momenta to the vector fields $X$ and $Y$; the flow of $K$ gives the geodesics in $\mathcal H$.  
The potential energy $U$ is the fundamental solution to the Heisenberg sub-Laplacian; see \cite{Folland}.  

The resulting system is at least \emph{partially integrable}, in that both the total energy $H$ and the angular momentum $p_{\theta}=xp_y-yp_x$ are conserved.  Moreover, the quantity 
\begin{equation}\label{J}
J=xp_x + yp_y +2zp_z
\end{equation}
generates the Carnot group dilations in $T^*\mathcal H$, which are given by 
\begin{equation}\label{dilation}
\delta_{\lambda}(x, y, z, p_x, p_y, p_z)=(\lambda x, \lambda y, \lambda^2 z, \lambda^{-1}p_x, \lambda^{-1}p_y, \lambda^{-2}p_z)
\end{equation}
for $\lambda>0$.
It satisfies 
\begin{equation}
\dot J =2H
\end{equation}
and is thus conserved if $H=0$.  
Due to the three integrals $H$, $p_{\theta}$, and $J$, the $H=0$ subsystem is integrable by the Liouville-Arnold theorem.
It is still
unknown whether the entire system is integrable.

The zero-energy subsystem has been and continues to be our main focus. In \cite{MS} we showed that periodic orbits must lie on the invariant hypersurface $\{H=0\}$, and that the dynamics are integrable there.  In \cite{CS} we showed that periodic orbits do exist and in \cite{DS} that they enjoy a rich symmetry structure; they realize every rational rotation number $j/k \in (0,1]$.  Periodic orbits must have $J=0$.  Here we investigate orbits with $J \neq 0$ as well.

As in \cite{DS}, much of this work was computer-assisted.  Running extensive careful numerical experiments led to many insights, including numerical verification of the statement of Theorem \ref{main thm} and the correct choice of dilation and time-reparametrization factors (see Section \ref{proof}),  which made the formal proof relatively simple.  
Our codebase can be found at \cite{Github}.

In Section \ref{main section} we state the main result and discuss the mild assumptions included and their potential removal.  In Section \ref{coords} we introduce a new coordinate system on $T^*\mathcal H$ specifically adapted for the Kepler-Heisenberg problem, which separates the symmetries of the system.  In Section \ref{thm, cors, pics} we state our result technically, derive some interesting consequences, and provide some pictures.  Finally, Section \ref{proof} contains the proof.

\section{Main Result and Discussion}  \label{main section}
We will call a Kepler-Heisenberg orbit $\gamma(t)=(x(t), y(t), z(t))$ an \emph{A-orbit} if it satisfies:
\begin{align*}
&H=0  \tag{A1}\label{H=0} \\
&\gamma(t) \  \text{does not meet the $z$-axis} \tag{A2}  \label{avoidzaxis} \\
&z(t)\ \text{has at least three zeros}.  \tag{A3}\label{3zeros} 
\end{align*}
The explanation for these conditions will constitute most of the the remainder of this section.  Our main result is the following, proved in Section \ref{proof}.
\begin{theorem}\label{main result}
Every Kepler-Heisenberg $A$-orbit is self-similar.  
\end{theorem}

\begin{figure}
\centering
\includegraphics[width=.8\textwidth]{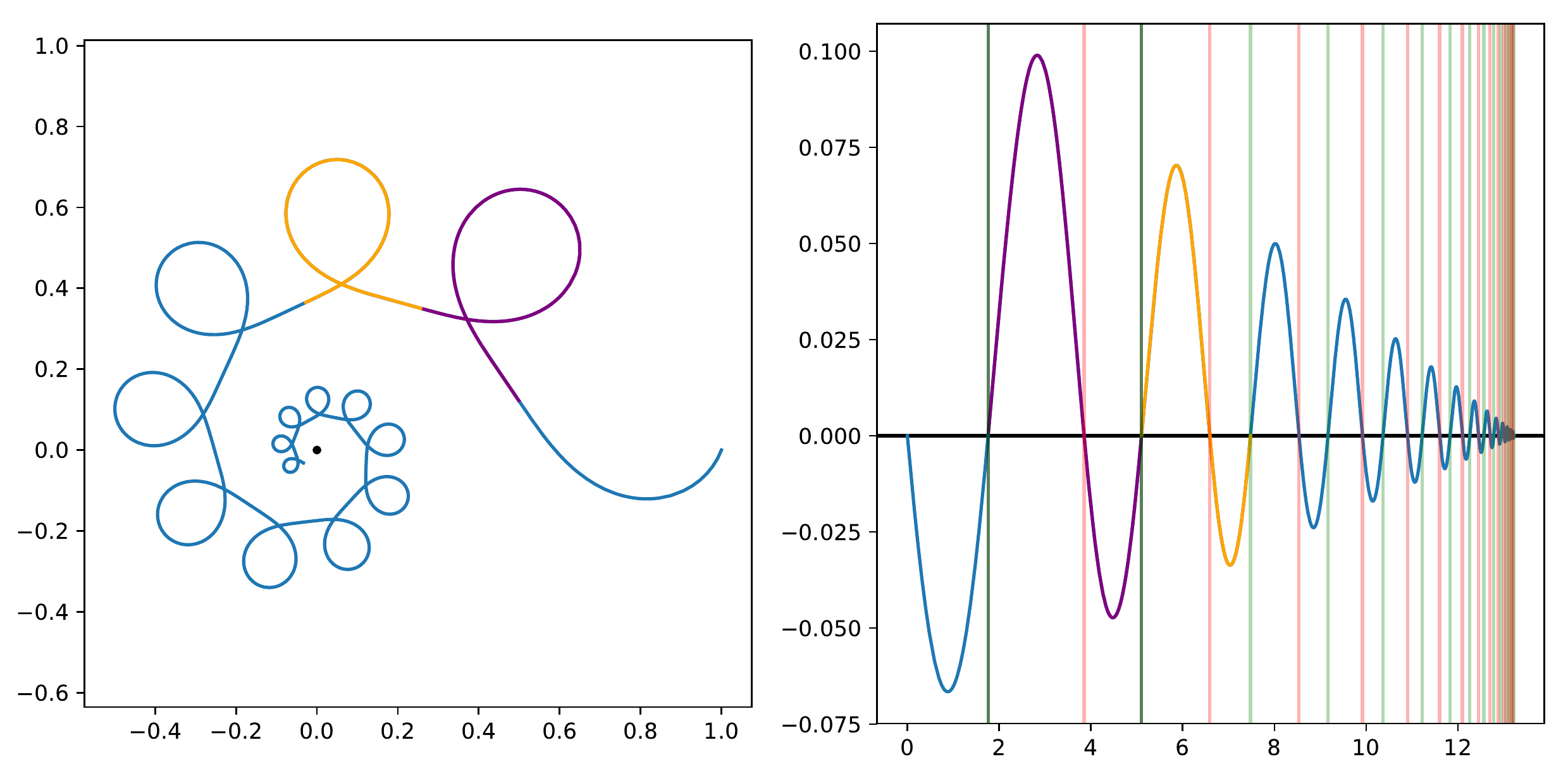}
\caption{A sample self-similar orbit.  Here and below, the left box shows the projection of the orbit to the $x,y$-plane.  The right box shows the $z$-coordinate over time, which is equal to the area traced out by this projection.  The black dot represents the sun, the blue curve represents the planet's orbit, the purple segment a sampled fundamental domain, and the orange segment the self-similar replication of the purple segment. }
\label{selfsim21}
\end{figure}

See Figure \ref{selfsim21} for an example. The self-similarity here involves a dilation (see (\ref{dilation})), a rotation, and a time-reparametrization.  A more precise statement can be found in Theorem \ref{main thm}.  
To the best of our knowledge these solutions represent the first occurrence of this type of self-similarity within 
solutions to ODEs of Hamiltonian type.

Note that due to the time-reparametrization,  this self-similarity is not true quasi-periodicity in the sense of ``periodic modulo a transformation of space".  However, zero-energy orbits are naturally stratified into three families, depending on whether $J$ is negative, positive, or zero.  Self-similarity for orbits with $J=0$ does not require a time-reparametrization, so we find that these orbits are genuinely quasi-periodic.  See Corollary \ref{quasi}.

We now explain the origin of conditions (\ref{H=0}) -- (\ref{3zeros}) above, and provide some conjectures regarding their necessity.

Condition (\ref{H=0}) is necessary as 
nearly everything we know about the Kepler-Heisenberg system lies within the $H=0$ subsystem.  As described in Section \ref{intro}, this subsystem is known to be integrable, and it contains all periodic orbits.

Condition (\ref{avoidzaxis}) is ostensibly a coordinate singularity.
The backbone of our proof of Theorem \ref{main result} is the correct choice of coordinate system, which separates the dilation and rotation group actions.  See Section \ref{coords}.   These charts exclude the $z$-axis.  However, this is not an artificial singularity: the $z$-axis in the Heisenberg group is the conjugate and cut locus.  Orbits with initial configuration on the $z$-axis can behave strangely.  For example, in \cite{MS} we discovered orbits of the form $(0, 0, c_1, 0, 0, c_2t)$ for constants $c_1, c_2$.  That is, a planet can sit at a stationary point on the $z$-axis with linearly growing momentum.  Note that such orbits satisfy $H<0$.  

Numerical analysis of zero-energy orbits meeting the $z$-axis shows two distinct types of behavior depending on the dilational momentum, with the bifurcation occuring at the critical value $|J|=\frac{1}{2\sqrt{\pi}}$ appearing in Proposition \ref{J bdd}.  
 If $|J|<\frac{1}{2\sqrt{\pi}}$ then the orbit is like a dilating figure-eight; as $J$ shrinks this family approaches the rotationally asymmetric periodic figure-eight orbit shown in Figure 3 of \cite{DS}.  These orbits are self-similar in the same way as the others considered in this paper, and we believe our proof can be modified to include this case.  However, if $|J|>\frac{1}{2\sqrt{\pi}}$ then the orbit is like a dilating helical Heisenberg geodesic.  These are self-similar in a different way; in particular, $z(t)$ is not oscillatory but monotonic. These planets escape the sun's gravitational pull in some sense.
 For $J>0$, for example, 
 the necessary escape momentum is determined by $J=xp_x+yp_y+2zp_z>\frac{1}{2\sqrt{\pi}}$; for an orbit starting on the $z$-axis we simply need $p_z>\frac{1}{4z\sqrt{\pi}}$.
Work in this area is ongoing; the singular geometry of the Heisenberg group along the $z$-axis means special care must be taken. 
Our numerical investigation suggests the following.
\begin{conjecture}\label{conj z axis selfsim}
Any zero-energy orbit meeting the $z$-axis is self-similar.
\end{conjecture}

Condition (\ref{3zeros}) represents our failure so far to prove our suspicion that $z(t)$ is oscillatory for any orbit satisfying (\ref{H=0}) and (\ref{avoidzaxis}).   
Here, \textit{oscillatory} is used in the sense of oscillation theory
: that $z(t)$ has infinitely many zeros.
We have numerically verified that this is the case for 2,500 uniformly sampled orbits.  
We conducted a comprehensive survey of the literature on Sturm comparison theorems and found no known version general enough for the second-order, inhomogeneous, nonlinear  ODE satisfied by $z(t)$, so the question remains open.
\begin{conjecture}\label{conj z osc}
For any zero-energy orbit not meeting the $z$-axis, $z(t)$ is oscillatory.
\end{conjecture}

There is a degenerate case: solutions for which $z$ is identically zero.    Proposition 3 of \cite{MS} shows that such orbits exist, but they are all necessarily lines through the origin in the $x,y$-plane.  These orbits are trivially self-similar, so Theorem \ref{main result} holds.  Note that if $z(t)=0$ on any open interval, then the orbit must be one of these lines through the origin in the plane by uniqueness of solutions to ODEs.  So among orbits satisfying condition (\ref{3zeros}), we will restrict our attention to those with three discrete zeros of $z$, and later in this paper we will refer to ``consecutive" zeros without reference to this degenerate case.

Note that we believe that Conjecture \ref{conj z osc} also holds for orbits which do meet the $z$-axis if $J<\frac{1}{2\sqrt{\pi}}$.  More importantly, Conjectures \ref{conj z axis selfsim} and \ref{conj z osc} say that we suspect our conditions 
(\ref{avoidzaxis}) and (\ref{3zeros}) can be removed from Theorem \ref{main result}.  That is, we suspect the following.

\begin{conjecture}
Every zero-energy orbit is self-similar.
\end{conjecture}

\section{New Coordinates}\label{coords}
Dilation by $\lambda>0$ and rotation about $z$-axis by $\varphi$ are important transformations of configuration space $\mathcal H$.  These naturally lift to phase space $T^*\mathcal H$ as the maps $\delta_{\lambda}$ and $\rho_{\varphi}$.  The former is given in (\ref{dilation}), and the latter by
$$\rho_{\varphi}(x, y, z, p_x, p_y, p_z) = \begin{bmatrix}
\cos \varphi & -\sin \varphi & 0 & 0 & 0 & 0 \\
\sin \varphi & \cos \varphi  & 0 & 0 & 0 & 0 \\
0 & 0 & 1 &  0 & 0 & 0 \\
0 & 0 & 0 & \cos \varphi & -\sin \varphi & 0 \\
0 & 0 & 0 & \sin \varphi & \cos \varphi & 0 \\
0 & 0 & 0 & 0 & 0 & 1 \\
\end{bmatrix}
\begin{bmatrix}
x \\ y \\ z \\ p_x \\ p_y \\ p_z
\end{bmatrix}.
$$

Both maps are symplectomorphisms of $T^*\mathcal H$, and they satisfy $H \circ \rho_{\varphi} = H$.
 and  $H \circ \delta_{\lambda} =\lambda^{-2} H$.  In the $H=0$ case, both are symmetries, so we look for explicit coordinates in which the corresponding moment maps are momentum coordinates.  
 Since the rotation and dilation actions commute, coordinates $(s, \theta, u)$ exist where dilations correspond to a shift in $s$, rotations to a shift in $\theta$, and $u$ is invariant under these actions. 
 For rotations, usual polar coordinates $(r, \theta)$ in the plane work, as $p_{\theta}$ is the desired conserved quantity of angular momentum.  However, finding coordinates with this property for $J$ rather than $p_{\theta}$ is much more difficult.  
 
We begin with position coordinates $(s, \theta, u)$ and insist that $p_s=J$. 
Here $s$ and $u$ are unknowns to be solved for, and $\theta = \arg(x,y)$ is the usual angle coordinate in the plane.  
 For convenience we use the notation $w=4z$ and $R=r^2$. 
Demanding canonical coordinates and $p_s=J$ yields the following system of first-order, nonlinear PDEs:
\begin{align*}
2R \left( \frac{\partial s}{\partial R} \frac{\partial u}{\partial w} - \frac{\partial u}{\partial R} \frac{\partial s}{\partial w} \right) &= \frac{\partial u}{\partial w}\\
2w \left( \frac{\partial s}{\partial R} \frac{\partial u}{\partial w} - \frac{\partial u}{\partial R} \frac{\partial s}{\partial w} \right) &= -\frac{\partial u}{\partial R}.
\end{align*}
Some algebra shows that $du$ must annihilate the radial vector field, so we choose $u=\arg(R,w)$.
This allows our system to be reduced to the first-order, linear PDE
$$R\frac{\partial s}{\partial R}+w\frac{\partial s}{\partial w}=\frac{1}{2} , $$  
which can be solved for $s$ by the method of characteristics.

Thus the following coordinates have the desired properties:
\begin{align*} 
s &= \frac{1}{4}\log ((x^2+y^2)^2+16z^2) \\
\theta &= \arg(x,y) \\ 
u &= \arg(x^2+y^2, 4z) \\ 
p_s &= xp_x+yp_y+2zp_z \\
p_{\theta} &= xp_y-yp_x \\ 
p_u &=  \frac{1}{4}p_z(x^2+y^2) -2z\frac{xp_x+yp_y}{x^2+y^2}.
\end{align*}
These are not global coordinates on $T^*\mathcal H$.  Clearly $s$ is singular at the origin, but this represents collision with the sun and is actually a singularity for our system (technically, our phase space should exclude the origin).  However, $\theta$ is not defined for points on the $z$-axis, so we must exclude this set from our chart.  The $z$-axis is special in the Heisenberg group: it is the cut and conjugate locus.  See Section \ref{main section} for further discussion.
Note that $u$ represents an inclination angle, with $\arg$ here taking values in $(-\pi/2, \pi/2)$, and that $z=0$ if and only if $u=0$.

By construction, these are especially nice coordinates to use, and they greatly simplify the expression of the rotation and dilation actions.
One easily checks by direct calculation that the following two results hold.

\begin{proposition}\label{coord maps}
We have that $(s, \theta, u, p_s, p_{\theta}, p_u)$ are canonical coordinates on $T^*\mathcal H$ and that  $p_s=J.$  Moreover, we have 
\begin{align*}
\rho_{\varphi}: &\theta \mapsto \theta + \varphi \\
\delta_{\lambda}: & s \mapsto s+\log \lambda
\end{align*}
with all other coordinates unaffected by these maps.
\end{proposition}

\begin{proposition}\label{coord H}
In $(s, \theta, u, p_s, p_{\theta}, p_u)$ coordinates, our Hamiltonian takes the form
$$
H(s, \theta, u, p_s, p_{\theta}, p_u)= \exp\left(-2s\right)\left(\frac{1}{2}\left[\begin{array}{ccc}
p_{s} & p_{\theta} & p_{u}\end{array}\right]\left[\begin{array}{ccc}
\cos u & \sin u & 0\\
\sin u & \sec u & 2\cos u\\
0 & 2\cos u & 4\cos u
\end{array}\right]\left[\begin{array}{c}
p_{s}\\
p_{\theta}\\
p_{u}
\end{array}\right]-\frac{1}{4\pi}\right).
$$
\end{proposition}



\section{Theorem, Consequences, Pictures}
\label{thm, cors, pics}

Recall that here we are only interested in orbits with $H=0$, for which the dilational momentum $J$ is an integral of motion.
In order to state the main theorem, Theorem \ref{main thm}, it remains to compute the appropriate dilation, rotation, and time-reparametrization factors.  Details, along with the proof, appear in Section \ref{proof}.  The coordinates $(s, \theta, u)$ and the maps $\rho_{\varphi}$ and  $\delta_{\lambda}$ are defined above in Section \ref{coords}.   Consequently, Theorem \ref{main result} is a paraphrasing of the following.

\begin{theorem}\label{main thm}
Assume $c:I \to T^*\mathcal H$ is a solution to Hamilton's equations on a maximal domain $I$.  Assume $H(c(t))=0$ for all $t$ and $c$ does not meet the $z$-axis.  Suppose $z(t)$ has consecutive zeros at $t_0, t_1, t_2$.
Let 
\begin{align*}
\lambda &= \exp (s(t_2) - s(t_0)) \\
\varphi &= \theta(t_2) - \theta(t_0)\\
\xi(t) &=\frac{1}{2}\log_{\lambda}\left(1-(t-t_0)\frac{1-\lambda^2}{t_2-t_0}  \right)  \\
\tau(t) &= t_0+(t_2-t_0)\frac{1-\lambda^{2\xi(t) + 2\lfloor \xi(t) \rfloor}}{1-\lambda^2}\\
\end{align*}
where we take $\xi(t)=\tau(t)=t$ if $\lambda=1$.
Then  for any time $t \in I$ we have
$$ c(t) = \rho_{\varphi} \circ \delta_{\lambda} \circ c \circ \tau(t).$$
\end{theorem}

We refer to the interval $[t_0, t_2]$ as a \emph{fundamental time domain}. 
The functions $\xi$ and $\tau$ are described more intuitively in Section \ref{constructing}; the symbol $\lfloor \xi(t) \rfloor$ denotes the floor function.
In Figure \ref{2cols} we show examples of self-similar orbits from each of the three families in Table \ref{table}.  In Figure \ref{4cols} we show an example of numerical verification of the self-similarity in phase space.
This type of numerical investigation preceded and facilitated both the statement and proof of Theorem \ref{main thm}.

\begin{figure}[h]
\centering 
\includegraphics[width=.78\textwidth]{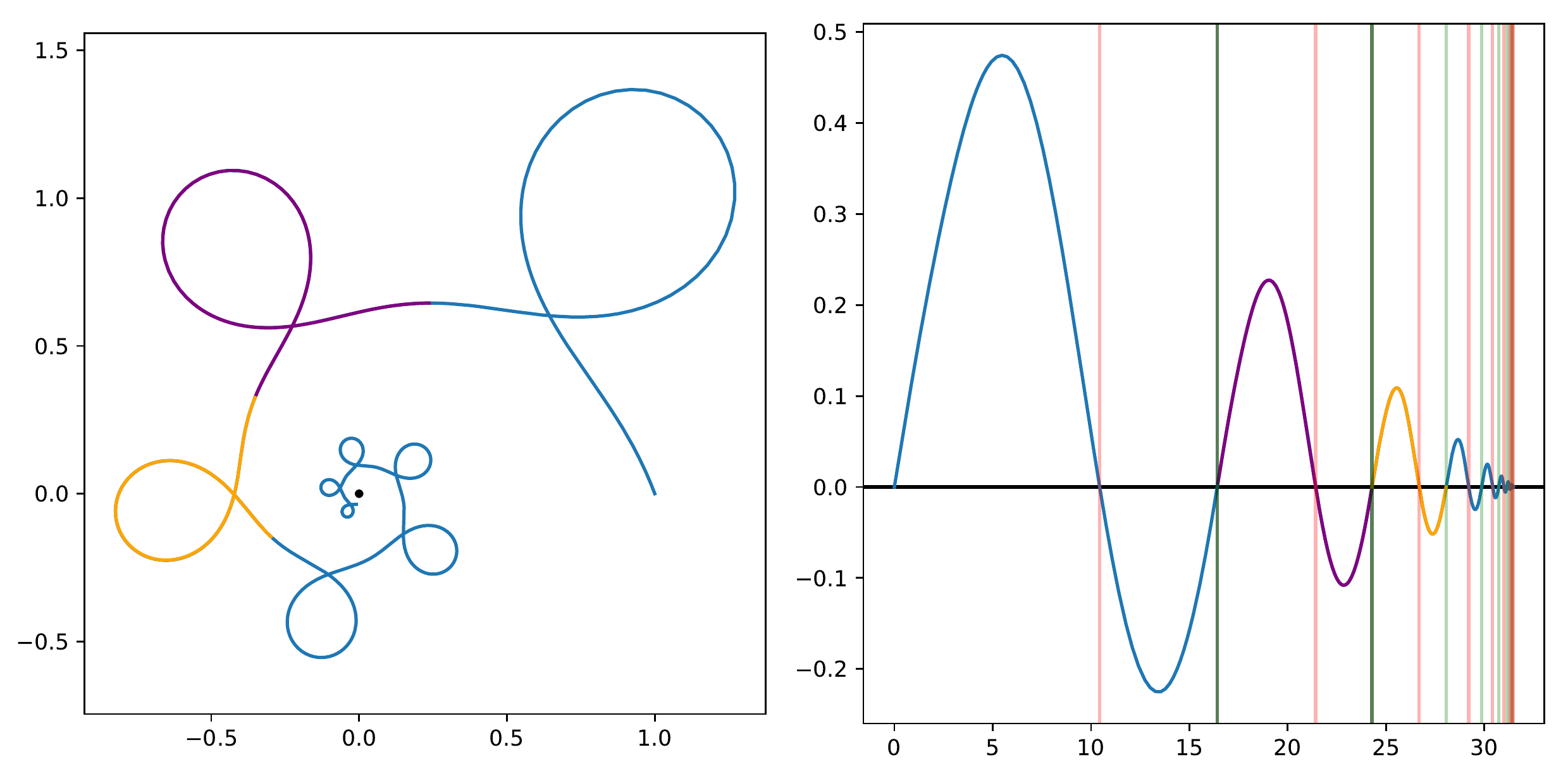} \\
\includegraphics[width=.78\textwidth]{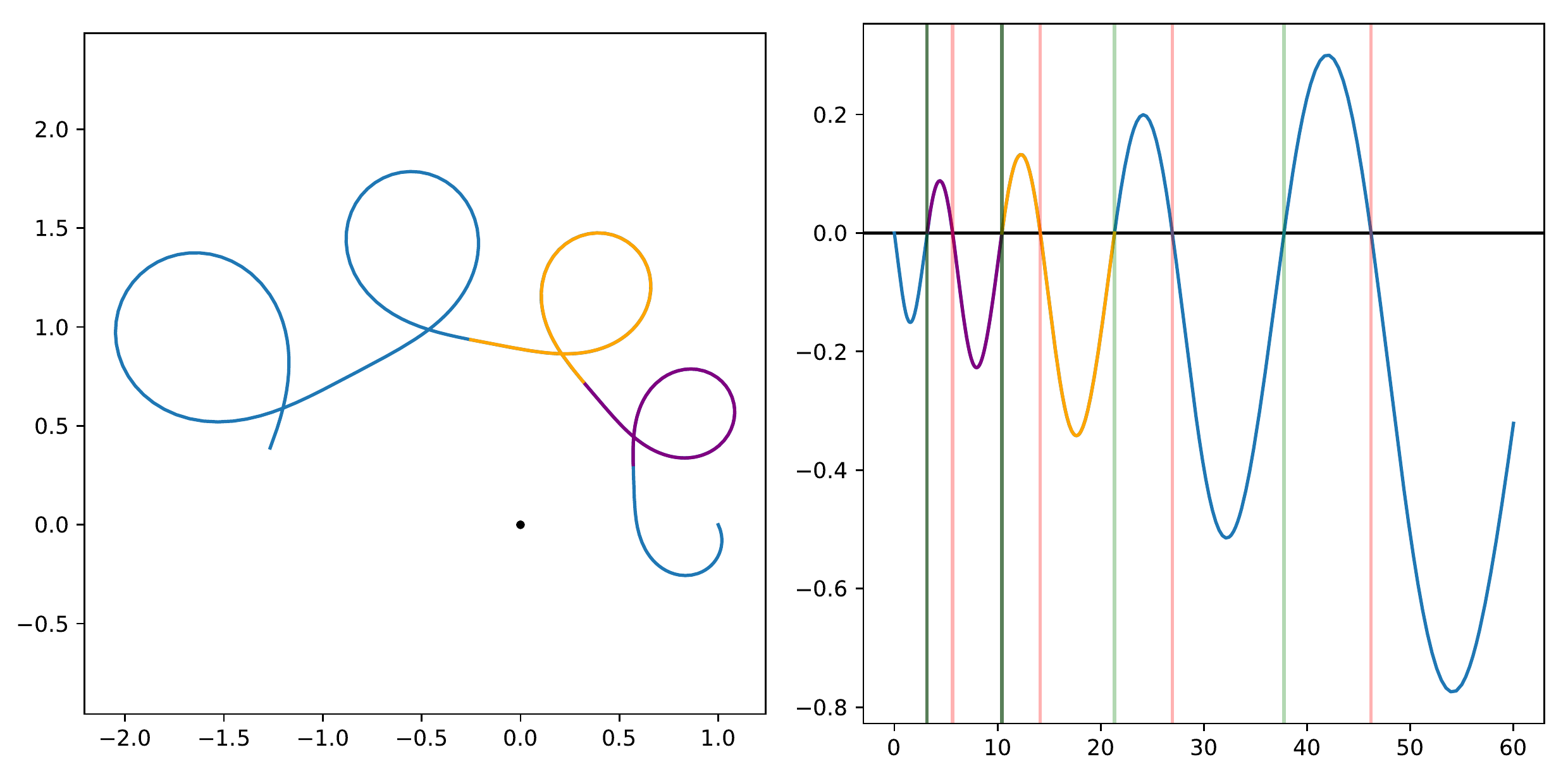} \\
\includegraphics[width=.78\textwidth]{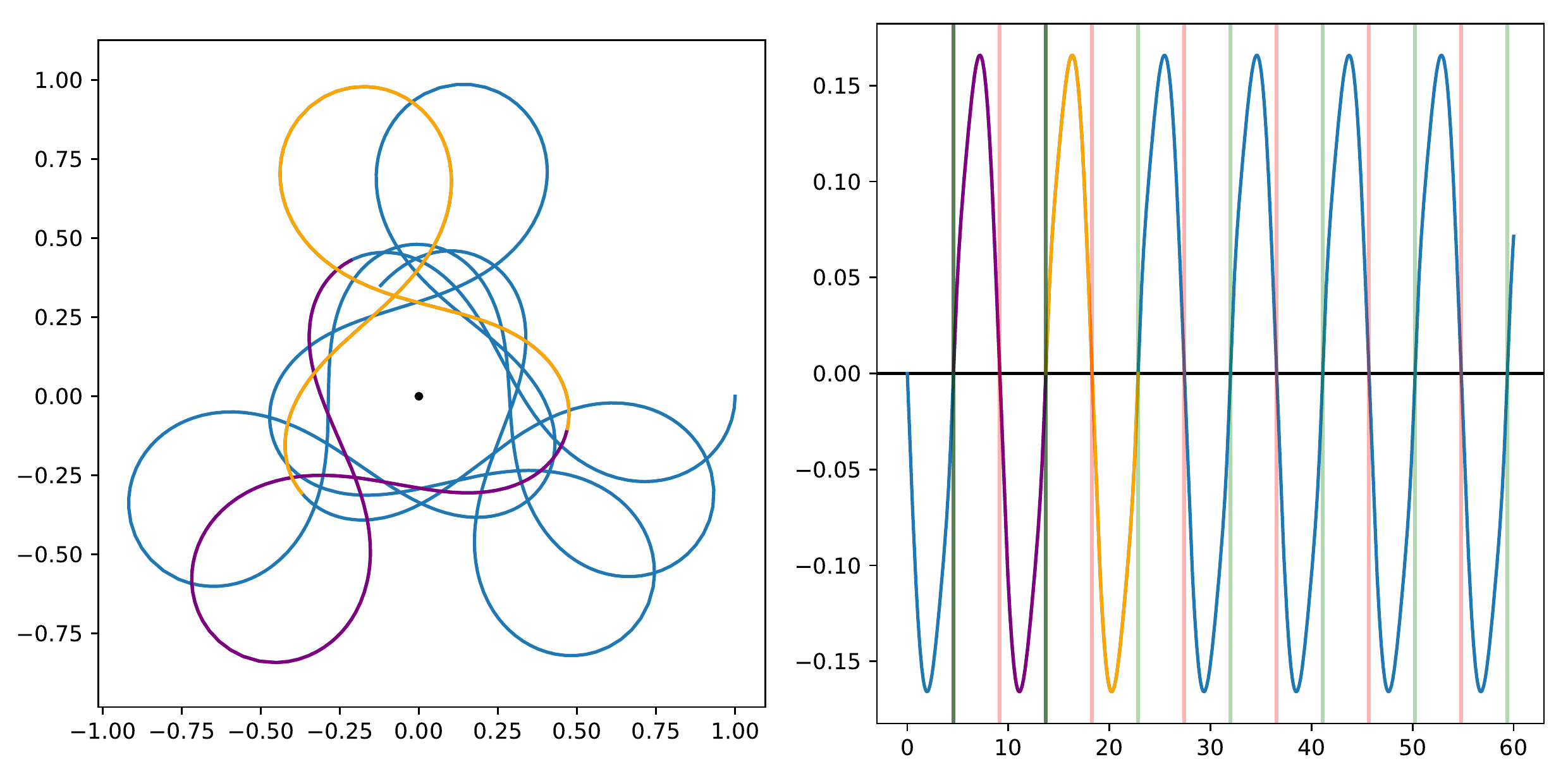}
\caption{Rows depict sample orbits with $\lambda<1$ , $\lambda>1$, and $\lambda=1$, respectively.  See Table \ref{table}.  Colors and columns as in Figure \ref{selfsim21}. 
Vertical lines show zeros of $z(t)$.  Fundamental time domains lie between three successive zeros.  Successive domains scale by $\lambda^2$.}
\label{2cols}
\end{figure}

We now give some interesting consequences of Theorem \ref{main thm};  recall from Section \ref{main section} that \textit{A-orbits} are those satisfying the hypotheses of the Theorem.

\begin{corollary}\label{quasi}
Every $A$-orbit with $J=0$ is quasi-periodic.
\end{corollary}

\begin{proof}
If $\lambda=1$, then $\tau(t)=t$ and there is no time-reparametrization involved, so the orbit is periodic modulo a transformation of space, namely  $\rho_{\varphi} \circ \delta_{\lambda}$.
\end{proof}

\begin{corollary}\label{finite time}
If an $A$-orbit collides, it does so in finite time.
\end{corollary}

\begin{proof}
By definition of $\tau$, consecutive fundamental time domains scale by $\lambda^2$.  By Theorem \ref{main thm}, a zero-energy orbit suffers a collision in future time if and only if $\lambda <1$.  For such an orbit, the time to collision is given by $\sum_{n=0}^{\infty} (t_2-t_0) \lambda^{2n}$, a convergent geometric series.  Thus the collision time is
\begin{equation}
t_{\text{col}}=t_0+\frac{t_2-t_0}{1-\lambda^2}.
\end{equation}
\end{proof}

\begin{corollary}\label{table cor}
 $A$-orbits stratify into the families described in Table \ref{table}.
\end{corollary}

\begin{proof}
This follows from Corollaries \ref{quasi} and \ref{finite time}, since reversing the flow of time inverts $\lambda$. 
\end{proof}

\begin{table}[h]
\caption{Summary of behavior of $A$-orbits}
\def\arraystretch{1.2}
  \begin{tabular}{ |  c | c | c | }
    \hline
    Dilational momentum & Dilation factor & Behavior of orbit \\ \hline 
    $J<0$ & $\lambda<1$ & Future collision, unbounded past \\ \hline 
    $J>0$ & $\lambda>1$ & Past collision, unbounded future \\ \hline
    $J=0$ & $\lambda=1$ & Periodic or quasi-periodic motion \\ \hline 
  \end{tabular}
  \label{table}
\end{table}


Note that the results in Table \ref{table} provide further evidence for our suspicion that invariant submanifolds are diffeomorphic to $\mathbb T^2 \times \mathbb R^+$.  The action of $J$ is non-compact, so we do not quite have invariant 3-tori, but rather cylinders.  When $J=0$, motion is restricted to a 2-torus.  


\begin{figure}
\centering
\includegraphics[width=.7\textwidth]{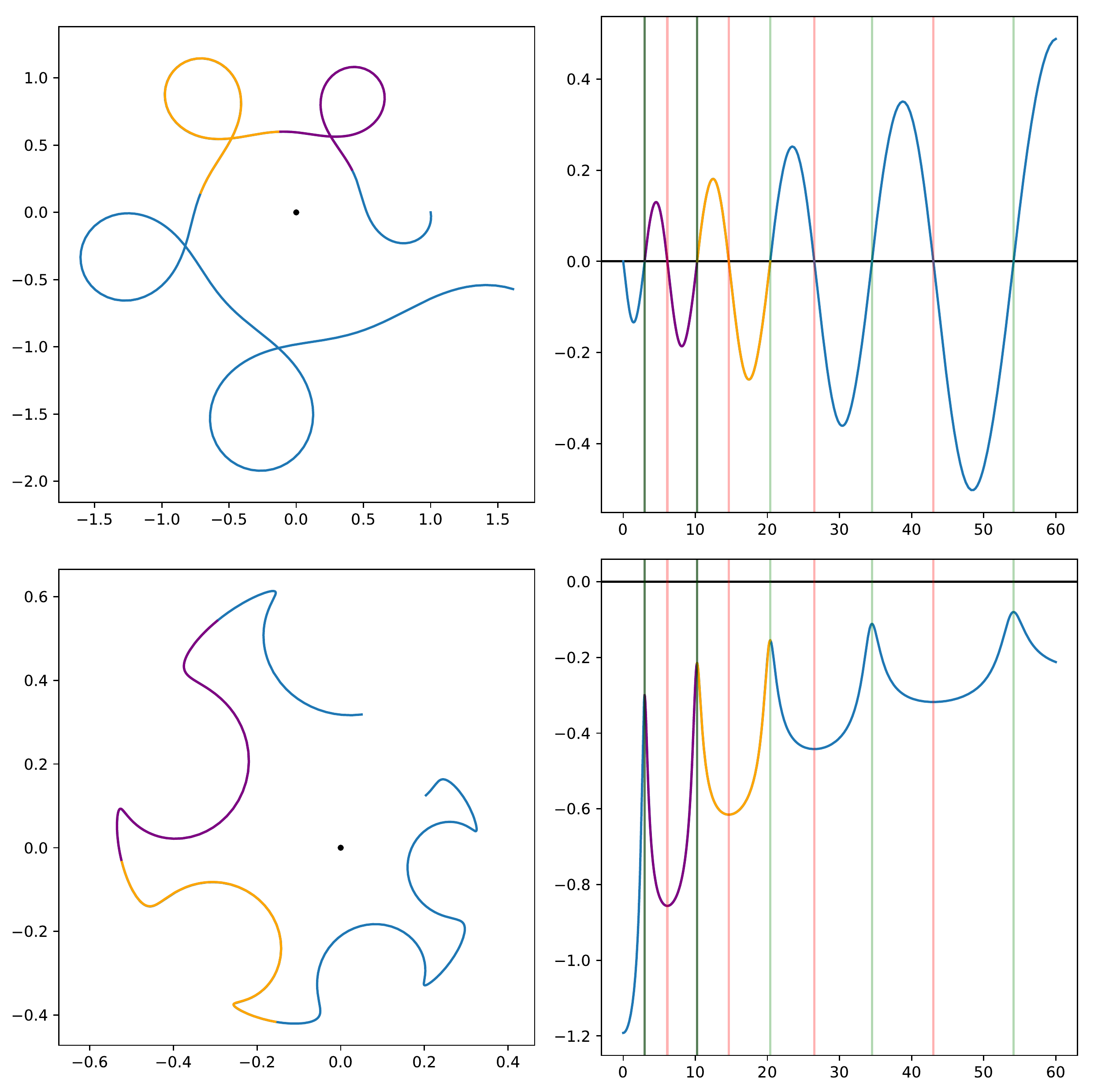}
\caption{Example of numerical verification of self-similarity hypothesis in phase space.  
A fundamental domain (purple) is replicated via self-similarity (orange). 
It matches visually, but we also compute numerical error to verify.
Top row depicts positions $(x,y)$ and $(t,z)$ as in previous figures.  Bottom row depicts momenta $(p_x, p_y)$ and $(t, p_z)$. 
}
\label{4cols}
\end{figure}

\section{Proof of Theorem}\label{proof}
First we briefly sketch the proof.  We extend an arbitrary $A$-orbit from a fundamental domain by gluing scaled, rotated, and reparametrized copies together.  We show this extended curve is continuous and satisfies Hamilton's equations.  Then by uniqueness of solutions to ODEs, the extended curve must be the one and only orbit which extends our original orbit.  Since the extended curve is self-similar by construction, the unrestricted original orbit must be as well.
We begin by proving some technical results.

Throughout this section we use the notation of Theorem \ref{main thm} and the coordinates from Section \ref{coords}.
Assume $c:I \to T^*\mathcal H$ is a solution to Hamilton's equations on a maximal domain $I$.  Assume $H(c(t))=0$ for all $t$.  Suppose $z(t)$ has consecutive zeros at $t_0, t_1, t_2$ and $c$ does not meet the $z$-axis.
Let 
\begin{align*}
\lambda &= \exp (s(t_2) - s(t_0)) \\
\varphi &= \theta(t_2) - \theta(t_0).
\end{align*}
The orbit $c$ (and consequently the constants $t_0, t_1, t_2, J=p_s, p_{\theta}, \lambda, \varphi$) will be fixed throughout; we will prove it is self-similar.

\subsection{Preliminary results}

\begin{lemma}\label{J bounds}
For any $A$-orbit $c(t)=(s(t), \theta(t), u(t), p_s(t), p_{\theta}(t), p_u(t))$,  we have
\[
-\tan\left(u\right)p_{\theta}-\sqrt{2\tan^{2}\left(u\right)p_{\theta}^{2}+\frac{1}{4\pi\cos\left(u\right)}}\leq p_{s}\leq-\tan\left(u\right)p_{\theta}+\sqrt{2\tan^{2}\left(u\right)p_{\theta}^{2}+\frac{1}{4\pi\cos\left(u\right)}}.
\]
\end{lemma}

\begin{proof}
This is simply a calculation, rewriting the $H=0$ condition, which is quadratic in the momenta coordinates.  Using Proposition \ref{coord H}, we find that $H=0$ is equivalent to
$$
p_{u} 
 =\frac{-\cos\left(u\right)p_{\theta}\pm\sqrt{\frac{1}{4\pi}\cos\left(u\right)-\left(\cos\left(u\right)p_{s}+\sin\left(u\right)p_{\theta}\right)^{2}}}{2\cos\left(u\right)}.
$$
Setting the discriminant $\Delta\left(p_{s}\right)=\frac{1}{4\pi}\cos\left(u\right)-\left(\cos\left(u\right)p_{s}+\sin\left(u\right)p_{\theta}\right)^{2}$ equal to zero gives
$$
p_s=-\tan\left(u\right)p_{\theta}\pm\sqrt{2\tan^{2}\left(u\right)p_{\theta}^{2}+\frac{1}{4\pi\cos\left(u\right)}}.
$$
Because $\Delta\left(p_{s}\right)$ is a negative-definite quadratic
function of $p_s$, it is the interior of these bounds that produces non-negative
values for $\Delta\left(p_{s}\right),$ and therefore
\[
-\tan\left(u\right)p_{\theta}-\sqrt{2\tan^{2}\left(u\right)p_{\theta}^{2}+\frac{1}{4\pi\cos\left(u\right)}}\leq p_{s}\leq-\tan\left(u\right)p_{\theta}+\sqrt{2\tan^{2}\left(u\right)p_{\theta}^{2}+\frac{1}{4\pi\cos\left(u\right)}}.
\]
Note that condition (\ref{avoidzaxis}) implies $\cos(u) \neq 0$ and the expression under the square root for these bounds is
always positive since $\cos\left(u\right)=\frac{R}{\sqrt{R^{2}+w^{2}}}$
and $R> 0$. 


\end{proof}

\begin{proposition}\label{J bdd}
For any $A$-orbit we have $|J|\leq \frac{1}{2\sqrt \pi}$.
\end{proposition}

\begin{proof}
Let $c(t)=(s(t), \theta(t), u(t), p_s(t), p_{\theta}(t), p_u(t))$ be such an orbit. Then at some time $t_0$ we have $z(t_0)=0$, so $u(t_0)=0$.  This implies $\tan(u(t_0))=0$ and $\cos(u(t_0))=1$.  Then by Lemma \ref{J bounds} we have  $-\sqrt{\frac{1}{4\pi}} \leq p_s(t_0) \leq \sqrt{\frac{1}{4\pi}}$.  But since $H(c)=0$, we have that $p_s=J$ is constant in time, which completes the proof.
\end{proof}

Proposition \ref{J bdd} is strong: it says there is a uniform bound on the dilational momentum of all $A$-orbits.  But this bound is also crucial in the proof of the next result, which says that the end point of such an orbit on a fundamental domain is a rotation and dilation of its start point.  This Lemma is the technical zenith of the proof of the main result; the rest will follow easily and naturally.

\begin{lemma}\label{endpts}
We have $c(t_2)= \rho_{\varphi} \circ \delta_{\lambda} \circ c(t_0)$.
\end{lemma}

\begin{proof}
The main idea here will be to show that 
\begin{equation}\label{p_u endpts}
p_u(t_0) = p_u(t_2);
\end{equation} 
 the other five coordinates are easy.  First note that $u(t_0)=u(t_2)=0$; that is, the endpoints here are characterized by lying in the $u=0$ plane (which is also the $z=0$ plane).  Therefore much of the subsequent analysis will be focused on the behavior of the orbit when $u=0$.
Imposing the two constraints $u=H=0$ and applying Proposition \ref{coord H} gives the following equation:

\begin{align*}
0 & =\left[\begin{array}{ccc}
p_{s} & p_{\theta} & p_{u}\end{array}\right]\left[\begin{array}{ccc}
\cos0 & \sin0 & 0\\
\sin0 & \sec0 & 2\cos0\\
0 & 2\cos0 & 4\cos0
\end{array}\right]\left[\begin{array}{c}
p_{s}\\
p_{\theta}\\
p_{u}
\end{array}\right]-\frac{1}{4\pi}\\
 & =4p_{u}^{2}+4p_{\theta}p_{u}+\left(p_{s}^{2}+p_{\theta}^{2}-\frac{1}{4\pi}\right).
 \end{align*}
Thus, $$
 p_{u}|_{u=0}  = -\frac{1}{2}	p_{\theta}\pm\frac{1}{2}\sqrt{\frac{1}{4\pi}-p_{s}^{2}}.$$
 
Now by Proposition \ref{J bdd}, we have $\frac{1}{4\pi}-p_s^2 \geq 0$, so this expression is always defined.
If $\frac{1}{4\pi}-p_s^2 = 0$, then $p_u|_{u=0} = -\frac{1}{2}p_{\theta}$, which is constant, so clearly (\ref{p_u endpts}) holds.
Now assume $\frac{1}{4\pi}-p_s^2 >0$.   
Let
$$
p_{u}^{\gamma}  = -\frac{1}{2}	p_{\theta}+\gamma\frac{1}{2}\sqrt{\frac{1}{4\pi}-p_{s}^{2}},
$$
where $\gamma\in\left\{ \pm1\right\} .$ 
From Hamilton's
equations,
$
\frac{du}{dt}\rvert_{u=0} 
 =\exp\left(-2s\right)\left(2p_{\theta}+4p_{u}|_{u=0}\right).
$
But $p_{u}|_{u=0}$ is equal to $p_{u}^{\gamma}$ for some
$\gamma$ that depends on $t.$ Denote this time-dependent $\gamma$
value by $\gamma\left(t\right).$ Thus
$$
\left. \frac{du}{dt}\right|_{u=0} 
 =\exp\left(-2s\right)\left(2p_{\theta}+4p_{u}^{\gamma\left(t\right)}\right)
  =2\gamma\left(t\right)\exp\left(-2s\right)\sqrt{\frac{1}{4\pi}-p_{s}^{2}}.
$$
Since $\frac{1}{4\pi}-p_{s}^{2}>0$, we have
$
\text{sign}\left(\left. \frac{du}{dt}\right|_{u=0} \right) 
 =\text{sign}\left(\gamma\left(t\right)\right),
$
showing that the solution curve $c\left(t\right)$ intersects the
plane $u=0$ transversally. Because $c$ is continuous,
the sign of $\left. \frac{du}{dt}\right|_{u=0} $ must alternate with each
intersection. Thus
$
\text{sign}\left(\gamma\left(t_{0}\right)\right)  =-\text{sign}\left(\gamma\left(t_{1}\right)\right)
 =\text{sign}\left(\gamma\left(t_{2}\right)\right),
$
and since $\gamma$ must take values in $\left\{ \pm1\right\} ,$
it follows that $\gamma\left(t_{0}\right)=\gamma\left(t_{2}\right),$
and therefore that
$
p_{u}\left(t_{0}\right)  =p_{u}\left(t_{2}\right).
$
Thus (\ref{p_u endpts}) holds in the case that $\frac{1}{4\pi}-p_s^2 >0$ as well.

Since $p_{s}$ and $p_{\theta}$ are conserved, we use (\ref{p_u endpts}), Proposition \ref{coord maps}, and the definitions of $\lambda$ and $\varphi$ to obtain
\begin{align*}
\rho_{\varphi}\circ\delta_{\lambda}\circ c\left(t_{0}\right) 
& =\rho_{\varphi}\circ\delta_{\lambda}\left[\begin{array}{cc}
s\left(t_{0}\right) & p_{s}\left(t_{0}\right)\\
\theta\left(t_{0}\right) & p_{\theta}\left(t_{0}\right)\\
u\left(t_{0}\right) & p_{u}\left(t_{0}\right)
\end{array}\right] \\
 & =\left[\begin{array}{cc}
s\left(t_{0}\right)+\log\exp\left(s\left(t_{2}\right)-s\left(t_{0}\right)\right) & p_{s}\left(t_{0}\right)\\
\theta\left(t_{0}\right)+\left(\theta\left(t_{2}\right)-\theta\left(t_{0}\right)\right) & p_{\theta}\left(t_{0}\right)\\
0 & p_{u}\left(t_{0}\right)
\end{array}\right]
  =\left[\begin{array}{cc}
s\left(t_{2}\right) & p_{s}\left(t_{2}\right)\\
\theta\left(t_{2}\right) & p_{\theta}\left(t_{2}\right)\\
u\left(t_{2}\right) & p_{u}\left(t_{2}\right)
\end{array}\right]
  =c\left(t_{2}\right).
\end{align*}

\end{proof}

\subsection{Constructing a self-similar trajectory}\label{constructing}

We now aim to paste together self-similar copies of $c$ restricted to $[t_0, t_2]$.  The main obstacle here is correctly reparametrizing time to account for the dilation factor $\lambda$.  If $\lambda=1$ then no such reparametrization is necessary, so we temporarily assume $\lambda \neq 1$.  
To this end, for any $\psi \in \mathbb R$, let 
$$\xi(t) =\frac{1}{2}\log_{\lambda}\left(1-(t-t_0)\frac{1-\lambda^2}{t_2-t_0}  \right)$$
\begin{align*}\label{tau}
\tau_{\psi} (t)&= 
t_0+(t_2-t_0)\frac{1-\lambda^{2\xi(t) + 2\psi}}{1-\lambda^2} \\
&=t_0 + (1-\lambda^{2\psi})\left( \frac{t_2-t_0}{1-\lambda^2}\right) + (t-t_0)\lambda^{2\psi}.
\end{align*}

\begin{lemma}\label{tau lemma}
The function $\tau_{\psi}$ enjoys the following properties.
\begin{enumerate}[(i)]
\item The map $\psi \mapsto \tau_{\psi}$ is a homomorphism from the additive group of real numbers to the group of invertible functions from $\mathbb R$ to $\mathbb R$ under composition.  That is, $\tau_{-\psi}=\tau_{\psi}^{-1}$ and $\tau_{\psi} \circ \tau_\eta = \tau_{\psi+\eta}$ and $\tau_0=\text{Id}_{\mathbb R}$.
\item  For any $k\in \mathbb Z$,  $\tau_k$ maps the fundamental domain $[t_0, t_2]$ to the appropriately scaled fundamental domain $k$ domains to the right. \label{shift}
\item If $\psi \neq 0$ then $\tau_{\psi}$ has exactly one fixed point, which is the collision time $t_{\text{col}}.$
\end{enumerate}  
\end{lemma}

\begin{proof}
This is a straightforward and tedious calculation. 
\end{proof}

We are finally in a position to extend $c$ from the domain $[t_0, t_2]$ by pasting together transformed copies.  We define the curve $C$ inductively, beginning with $C(t)=c(t)$ for all $t \in [t_0, t_2]$.  We then define $C$ on the next fundamental domain as 
$$C(t) = \rho_{\varphi} \circ \delta_{\lambda} \circ c \circ \tau_1(t)
\qquad \text{for all} \  t \in [t_2, t_2 + \lambda^2(t_2-t_0)].$$
We can continue to the next fundamental domain by letting
$$C(t) = \rho_{\varphi} \circ \delta_{\lambda} \circ c \circ \tau_2(t)
\qquad \text{for all} \  t \in [t_2 + \lambda^2(t_2-t_0), t_2 + \lambda^2(t_2-t_0)+\lambda^4(t_2-t_0)].$$
In general, for $k\in \mathbb Z^+$, we have 
\begin{equation}\label{pos k}
C(t) = \rho_{\varphi} \circ \delta_{\lambda} \circ c \circ \tau_k(t)
\qquad \text{for all} \  t \in \left[t_0 + (t_2-t_0)\sum_{i=0}^{k-1}\lambda^{2i}, t_0 + (t_2-t_0)\sum_{i=0}^{k}\lambda^{2i}\right].
\end{equation}

Similarly, we can define $C$ for past times, on the fundamental domain prior to $[t_0, t_2]$ as 
$$C(t) = \rho_{\varphi} \circ \delta_{\lambda} \circ c \circ \tau_{-1}(t)
\qquad \text{for all} \  t \in [t_0 - \lambda^{-2}(t_2-t_0), t_0],$$
and generally for $k\in \mathbb Z^-$, by
\begin{equation}\label{neg k}
C(t) = \rho_{\varphi} \circ \delta_{\lambda} \circ c \circ \tau_k(t)
\qquad \text{for all} \  t \in \left[t_2 - (t_2-t_0)\sum_{i=0}^{k}\lambda^{2i}, t_2 - (t_2-t_0)\sum_{i=0}^{k+1}\lambda^{2i}\right].
\end{equation}
Note that these expressions are still valid for the case of $\lambda=1$ if we set $\tau_{\psi}(t)=t$ for all $\psi$.
Thus we piece-wise construct the curve $C$, whose domain is $(-\infty,t_{\text{col}})$ if $\lambda<1$, $(t_{\text{col}}, \infty)$ if $\lambda>1$, and all real numbers if $\lambda =1$. 

The functions $\xi$, $\tau_{\psi}$ (above) and $\tau$ (Theorem \ref{main thm}) require some explanation.  Intuitively, $\xi$ dilates the original time $t$ so that $z(\xi)$ is periodic with constant period $T=t_2-t_0$, rather than having ``periods" forming the geometric sequence $T, \lambda^2 T, \lambda^4 T, \dots$.  For an integer $k$, the function $\tau_k$ shifts the fundamental domain $[t_2, t_0]$ by $k$ fundamental domains to the right, as can be seen in Equations (\ref{pos k}) and (\ref{neg k}) above.  The function $\tau$ in Theorem \ref{main thm} combines the two in order to avoid the inductive construction of $C$ given in this section and state the main theorem in a more concise and self-contained manner.  More precisely, the floor $\lfloor \xi(t) \rfloor$ computes the integer $k$ that represents how many fundamental domains to the right of $[t_0, t_2]$ we must shift to contain the time $t$.  That is,
$$ \tau = \tau_{\lfloor \xi \rfloor}.$$

\begin{proposition}
The curve $C$ is continuous.
\end{proposition}

\begin{proof}
This follows from Proposition \ref{endpts} and Lemma \ref{tau lemma} (\ref{shift}).  By construction, the pieces of $C$ agree at the endpoints (and $C$ is clearly continuous elsewhere).
\end{proof}

\begin{proposition}
The curve $C$ is smooth and satisfies Hamilton's equations.
\end{proposition}

\begin{proof}
This holds on $[t_0, t_2]$ by assumption: $C=c$ there and $c$ was assumed to be a solution.  
The maps $\rho_{\varphi}$ and $\delta_{\lambda}$ are symmetries of the $H=0$ system ($\rho_{\varphi}$ is a symmetry of the full system) and commute with the flow of the Hamiltonian vector field.  That is, they preserve solutions.  This can also be seen in Proposition \ref{coord maps}, as both simply translate a single coordinate in the appropriate chart.  
 
Now $\tau = \tau_{\lfloor \xi \rfloor}$ shifts the fundamental domain $[t_0, t_2]$ and dilates it by a factor of $\lambda^2$.  The shift simply gives us the correct domain for $C$.  In \cite{MS}, we showed that if a curve $c$ satisfies Hamilton's equations, then so does $\delta_{\lambda}(c(\lambda^{-2}t))$.  In fact, we showed more generally that a result like this holds for any homogeneous potential, and leads to a version of Kepler's third law.  Here, we have that $H$ is homogeneous of degree two, as $H\circ \delta_{\lambda}=\lambda^{-2}H$.

Thus each piece of $C$ is a solution.  Since $C$ is continuous in phase space we find that $C$ is a solution on all its domain.  Since $H$ is smooth away from the origin, it has a smooth Hamiltonian vector field, which has smooth solutions, showing that $C$ is in fact smooth.

\end{proof}

\subsection{Proof of Theorem \ref{main thm}}
All the hard work is done, and Theorem \ref{main thm} now follows from the Picard-Lindel\"{o}f Theorem.  
Our Hamiltonian vector field is Lipschitz at the endpoints of the fundamental domains since it is continuous on pre-compact neighborhoods of these points.  
Thus $C$ is the unique solution to Hamilton's equations and extends $c$ to its maximal domain $I$ (to future or past collision).

\subsection*{Acknowledgments}
The authors are grateful to Richard Montgomery for many valuable conversations.  The first author would like to thank MSRI for the privilege of taking part in the Fall 2018 Hamiltonian Systems semester, during which the bulk of this project was completed.
















\bigskip

\end{document}